\documentclass{amsart}
\usepackage{url}
\usepackage{amssymb}
\usepackage{enumerate}
\usepackage{braket}

\makeatletter
\@namedef{subjclassname@2010}{%
  \textup{2010} Mathematics Subject Classification}
\makeatother

\theoremstyle{theorem}
\newtheorem{thm}{Theorem}
\newtheorem{prethm}{Theorem}

\newtheorem{lem}{Lemma}

\theoremstyle{definition}

\newtheorem{rem}{Remark}

\DeclareMathOperator*{\psum}{{\sum}^\ast}
\DeclareMathOperator*{\dsum}{\sum\sum}
\DeclareMathOperator*{\msum}{\sum\cdots\sum}
\DeclareMathOperator*{\puni}{{\coprod}^\ast}

\newcommand{\ppmod}[1]{\ (\mathrm{mod}\ {#1})}

\newcommand{\M}{\mathfrak{M}}
\newcommand{\m}{\mathfrak{m}}
\newcommand{\Maq}{\mathfrak{M}_{a,q}}
\newcommand{\Naq}{\mathfrak{M}'_{a,q}}

\renewcommand{\phi}{\varphi}
\renewcommand{\epsilon}{\varepsilon}
\renewcommand{\Re}{\mathrm{Re}}

\begin{document}
\title[On prime vs. prime power pairs]
{On prime vs. prime power pairs}

\author[Y. Suzuki]{Yuta Suzuki}



\maketitle

\begin{abstract}
In this paper, we consider pairs of a prime and a prime power
with a fixed difference.
We prove an average result on the distribution of such pairs.
This is a partial improvement of the result of Bauer (1998).
\end{abstract}

\section{Introduction}
In his famous address at the 5th International Congress of Mathematicians,
Landau \cite{Landau} listed four problems in prime number theory,
which are now called Landau's problems.
These problems are:
\begin{enumerate}
\item
Does the function $u^2+1$ represent infinitely many primes for integers $u$?
\item
Does the equation $m=p+p'$ have for any even $m$ a solution in primes?
\item
Does the equation $2=p-p'$ have infinitely many solutions in primes?
\item
Does at least one prime exist between $n^2$ and $(n+1)^2$
for any positive integer $n$?
\end{enumerate}
The present paper is related to the first three problems from Landau's list.

Landau's third problem is well-known as the twin prime problem.
Let
\begin{equation}
\label{twin_prime_counting}
\Psi(X,h)
=
\sum_{n\le X}\Lambda(n)\Lambda(n+h),
\end{equation}
where $h$ is a positive integer and $\Lambda(n)$ is the von Mangoldt function.
This function $\Psi(X,h)$ counts the number of twin prime pairs,
i.e. prime pairs $(p,p')$ satisfying the twin prime equation
\begin{equation}
\label{twin_prime_eq}
p'=p+h,
\end{equation}
which slightly generalizes the twin prime problem.
Although Landau confessed
that his problems seem unattackable at the state of science at his time,
Hardy and Littlewood introduced a new method, which is called now the circle method,
and gave some important attacks against problems on prime numbers.
By applying their method formally,
Hardy and Littlewood found an hypothetical asymptotic formula
\begin{equation}
\label{HL_asymp_twin}
\Psi(X,h)
=
\mathfrak{S}(h)X
+
(\text{Error})
\end{equation}
for even $h$,
where $\mathfrak{S}(h)$ is the singular series for the twin prime problem defined by
\[
\mathfrak{S}(h)
=
\prod_{p|h}\left(1+\frac{1}{p-1}\right)
\prod_{p\nmid h}\left(1-\frac{1}{(p-1)^2}\right).
\]
In this note, we call this type of hypothetical asymptotic formula
\emph{the Hardy-Littlewood asymptotic formula}.
Note that the Bateman-Horn conjecture \cite{Bateman_Horn}
gives a much wider picture on the distribution of prime numbers.
Since $\mathfrak{S}(h)\gg1$,
the Hardy-Littlewood asymptotic formula (\ref{HL_asymp_twin}) gives
a positive answer to the twin prime problem.
Unfortunately,
any rigorous proof of (\ref{HL_asymp_twin}) seems quite far
from our current state of science.
However,
some average behavior of $\Psi(X,h)$
have been obtained by many researchers.
As for the twin prime problem,
Mikawa \cite{Mikawa_twin} or Perelli and Pintz \cite{PP_Goldbach} obtained
the current best result:
\begin{prethm}[{Mikawa \cite{Mikawa_twin}, Perelli and Pintz \cite{PP_Goldbach}}]
\label{MPP_thm}
Let $X,H,A\ge2$, and $\epsilon>0$.
Assume
\[
X^{1/3+\epsilon}\le H\le X.
\]
Then we have
\[
\Psi(X,h)=\mathfrak{S}(h)X+O(XL^{-A})
\]
for all but $\ll HL^{-A}$ even numbers $h\in [1,H]$.
\end{prethm}
\noindent
Since the original twin prime problem is the case $h=2$,
we are interested in restricting $h$ to some small neighborhood of $h=2$.
Namely, our goal is to obtain the result under the situation
``the larger $X$ with the smaller $h$''.
In this note, we consider this kind of average results
for the Hardy-Littlewood asymptotic formulas.

We next consider Landau's first problem.
Let
\begin{equation}
\label{HL_counting}
\Psi_k(X,h)=\sum_{n^k\le X}\Lambda(n^k+h),
\end{equation}
where $k\ge2$ is a positive integer.
This function counts the number of pairs $(n^k, p)$
satisfying the equation
\begin{equation}
\label{HL_eq}
p=n^k+h,
\end{equation}
which generalizes Landau's first problem.
Note that if the polynomial $X^k+h\in\mathbb{Q}[X]$ is reducible,
then the equation (\ref{HL_eq}) has only a finite number of solutions.
Thus we introduce
\[
\mathbf{Irr}_k=
\Set{h\in\mathbb{N}|\text{$X^k+h$ is irreducible over $\mathbb{Q}$}}.
\]
As for this equation,
the Hardy-Littlewood asymptotic formula is given by
\begin{equation}
\label{HL_asymp_HL}
\Psi_k(X,h)
=
\mathfrak{S}_k(h)X^{1/k}
+
(\text{Error})
\end{equation}
for $h\in\mathbf{Irr}_k$,
where the singular series $\mathfrak{S}_k(h)$ is given by
\[
\mathfrak{S}_k(h)
=
\prod_p\left(1-\frac{r_k(h,p)-1}{p-1}\right),
\]
\[
r_k(h,p)=
\left|
\Set{x\ppmod{p}|x^k+h\equiv0\ppmod{p}}
\right|.
\]
The average result for this problem is obtained recently
by \cite{Baier_Zhao_1, Baier_Zhao_2, Foo_Zhao}.
We note that as for the ``conjugate'' equation
\[
N=p+n^k,
\]
some results were obtained earlier
by \cite{Mikawa_square, PP_square, Perelli_Zaccagnini},
and it seems straightforward to apply these earlier work
to the function $\Psi_k(X,h)$
and give the same result as in \cite{Baier_Zhao_2}
or even better results than those of \cite{Baier_Zhao_1, Foo_Zhao}.
We have to mention that
the interesting method used in \cite{Baier_Zhao_2}
is completely different from the earlier work.
Namely, Baier and Zhao showed that Linnik's dispersion method is sometimes applicable
to our problem, which is originally attacked by the circle method in earlier work.
As a result of these work, the current best result is:
\begin{prethm}[Perelli and Zaccagnini \cite{Perelli_Zaccagnini}]
\label{Perelli_Zaccagnini}
Let $X,H,A\ge2$, and $\epsilon>0$.
Assume
\[
X^{1-1/k+\epsilon}\le H\le X.
\]
Then we have
\[
\Psi_k(X,h)=\mathfrak{S}_k(h)X^{1/k}+O(X^{1/k}L^{-A})
\]
for all but $\ll HL^{-A}$ integers $h\in [1,H]\cap\mathbf{Irr}_k$.
\end{prethm}

In this paper,
we consider a kind of mixture of the above two problems.
Namely,
we consider the ``prime vs. prime power'' pairs $(p^k,p')$ satisfying the equation
\begin{equation}
\label{Hua_eq}
p'=p^k+h
\end{equation}
which can be regarded as
a mixture of equations (\ref{twin_prime_eq}) and (\ref{HL_eq}).
We introduce the sets
\[
\mathbb{H}_k^\mathrm{local}=
\Set{h\in\mathbb{N}|
\forall p\text{\,:\,prime},\,
(p-1)|k\Rightarrow h\not\equiv-1\ppmod{p}},\]
\[
\mathbb{H}_k=
\mathbb{H}_k^\mathrm{local}\cap\mathbf{Irr}_k.
\]
As for this equation (\ref{Hua_eq}),
the counting function is given by
\[
\Psi_k^\ast(X,h)=
\sum_{n^k\le X}\Lambda(n)\Lambda(n^k+h),
\]
and the Hardy-Littlewood asymptotic formula takes the form
\begin{equation}
\label{HL_asymp_Hua}
\Psi_k^\ast(X,h)
=
\mathfrak{S}_k^\ast(h)X^{1/k}
+
(\text{Error})
\end{equation}
for $h\in\mathbb{H}_k$,
where
\[
\mathfrak{S}_k^\ast(h)
=
\prod_{p|h}\left(1+\frac{1}{p-1}\right)
\prod_{p\nmid h}\left(1-\frac{(w_k(h,p)-1)p+1}{(p-1)^2}\right),
\]
\begin{equation}
\label{w_def}
w_k(h,p)
=
\left|
\Set{x\ppmod{p}|x^k+h\equiv0\ppmod{p},\ (x,p)=1}
\right|.
\end{equation}
As for the equation (\ref{Hua_eq}),
Liu and Zhan \cite{Liu_Zhan} obtained a result for the case $k=2$,
and Bauer \cite{Bauer} generalized their result to general $k$:
\begin{prethm}[Bauer \cite{Bauer}]
\label{Bauer_full}
Let $X,H,A\ge2$, and $\epsilon>0$.
Assume
\[
X^{1-1/2k+\epsilon}\le H\le X.
\]
Then we have
\[
\Psi_k^\ast(X,h)=\mathfrak{S}_k^\ast(h)X^{1/k}+O(X^{1/k}L^{-A})
\]
for all but $\ll HL^{-A}$ integers $h\in [1,H]\cap\mathbb{H}_k$.
\end{prethm}
\noindent
We remark that
the results in \cite{Bauer, Liu_Zhan} are stated with the conjugate equation
\begin{equation}
\label{conjugate_Hua}
N=p^k+p'.
\end{equation}
The aim of this paper is to improve this result of Bauer.
In particular, we have
\begin{thm}
\label{main_thm}
Let $X,H,A\ge2$, and $\epsilon>0$.
Assume
\[
X^{1-1/k+\epsilon}\le H\le X.
\]
Then we have
\[
\Psi_k^\ast(X,h)=\mathfrak{S}_k^\ast(h)X^{1/k}+O(X^{1/k}L^{-A})
\]
for all but $\ll HL^{-A}$ integers $h\in [1,H]\cap\mathbb{H}_k$.
\end{thm}
\noindent
As it can be easily predicted,
our method is also applicable to the conjugate equation (\ref{conjugate_Hua}).
Moreover our method gives
a minor variant of the proof of Theorem \ref{Perelli_Zaccagnini},
i.e. our method is applicable to somewhat broader context
than the method in \cite{Perelli_Zaccagnini}.
Although our method gives an improvement of Theorem \ref{Bauer_full},
it has some disadvantage compared with \cite{Bauer, Liu_Zhan, Perelli_Zaccagnini}.
Briefly speaking, our method can not be applied to the restricted counting function.
See the last section of this paper.

Our method is inspired
by the work \cite{Bauer, Mikawa_three, Mikawa_Peneva, PP_Goldbach}.
In particular, the idea of Mikawa \cite{Mikawa_three}
or its variant of Mikawa and Peneva \cite{Mikawa_Peneva}
gives our strategy for the treatment of the minor arcs.
In these work \cite{Mikawa_three, Mikawa_Peneva},
the minor arc estimates are reduced in an efficient way
to some Vinogradov-type estimates for sums over prime numbers.
In our case,
we shall reduce the minor arcs estimate for the equation (\ref{Hua_eq})
to the minor arc estimate for the twin prime equation (\ref{twin_prime_eq})
which is given
by Mikawa \cite{Mikawa_twin}
or
by Perelli and Pintz \cite{PP_Goldbach}.
See Sections 6 and 7.

\section{Notation}
We shall use the following notation.
Throughout the letters
$\alpha, \eta$
denote real numbers,
$X,Y,H,U,M,P,Q,R,A,B,\epsilon$
denote positive real numbers,
$m,n,d,h,u,N$
denote integers,
$k\ge2$ denotes a positive integer,
$p$ denotes a prime number,
and $L=\log X$.
For any real number $\alpha$,
let $e(\alpha)=e^{2\pi i\alpha}$.
The arithmetic function
$\phi(n)$ denotes the Euler totient function,
$\Lambda(n)$ denotes the von Mangoldt function,
$\mu(n)$ denotes the M\"obius function,
and $\tau_k(n)$ is defined by
\[
\tau_k(n)=\sum_{d_1\cdots d_k=n}1.
\]
The letters $a,q$ denote positive integers satisfying
$(a,q)=1$ and
the expressions
\[\psum_{a\ppmod{q}},\quad\puni_{a\ppmod{q}}\]
denote a sum and a disjoint sum over all reduced residues $a\ppmod{q}$
respectively.

We use the following trigonometric polynomials:
\[
S_1(\alpha)=\sum_{n\le 2X}\Lambda(n)e(n\alpha),\quad
V_1(\eta)=
\sum_{n\le 2X}e(n\eta),
\]
\[
S_k(\alpha)=\sum_{XL^{-4kA}<n^k\le X}\Lambda(n)e(n^k\alpha),\quad
V_k(\eta)=
\frac{1}{k}
\sum_{XL^{-4kA}<n\le X}n^{1/k-1}e(n\eta),
\]
for $k\ge2$.
We introduce the following complete exponential sums
\[C_k(a,q)=\psum_{m\ppmod{q}}e\left(\frac{am^k}{q}\right),\quad
A_k(n,q)=\psum_{a\ppmod{q}}\overline{C_k(a,q)}e\left(-\frac{an}{q}\right).\]
Note that if $(a,q)=1$,
then the exponential sum $C_1(q,a)$
is reduced to the M\"obius function $\mu(q)$.
Then we introduce the remainder terms
\[
R_k(\eta,a,q)=
S_k\left(\frac{a}{q}+\eta\right)-\frac{C_k(a,q)}{\phi(q)}V_k(\eta)
\]
and the truncated singular series
\[
\mathfrak{S}_k^\ast(h,P)
=
\sum_{q\le P}\frac{\mu(q)A_k(h,q)}{\phi(q)^2}.
\]
We shall use the constant
$
K=2^{k-1}.
$
We assume $B\ge B_0(k,A)$,
where $B_0(k,A)$ is some positive constant depends only on $k$ and $A$.
The implicit constants may depend on $k,A,B,\epsilon$.
We assume $A\ge k$ without loss of generality.

\section{The Farey dissection}
As usual,
we deduce Theorem \ref{main_thm}
from the following $L^2$-estimate:
\begin{thm}
\label{main_mean}
Let $X,H,A,B\ge2$, $U\ge0$, $\epsilon>0$, and $P=L^B$.
Assume
\[
X^{1-1/k+\epsilon}\le H\le X,\quad
0\le U\le X.
\]
Then for sufficiently large $B\ge B_0(k,A)$, we have
\begin{equation}
\label{mean_error}
\sum_{U<h\le U+H}
\left|\Psi_k^\ast(X,h)-\mathfrak{S}_k^\ast(h,P)X^{1/k}\right|^2\ll HX^{2/k}L^{-4A},
\end{equation}
where the implicit constant depends on $k,A,B,\epsilon$.
\end{thm}

We start the proof of Theorem \ref{main_mean}.
We can assume that $U$ and $H$ are positive integers
since the contribution of some bounded variation of $U$ or $H$ to (\ref{mean_error})
is at most\footnote{
See the estimate (\ref{truncated_estimate}) in Section 5.
}
\[\ll X^{2/k}L^{2k}\ll HX^{2/k}L^{-3A}.\]
Moreover,
notice that it is sufficient to prove Theorem \ref{main_mean} for the case
\[X^{1-1/k+\epsilon}\le H\le X^{4/5},\]
which makes the proof of Theorem \ref{90_moment} simpler.

By the orthogonality of additive characters we have
\begin{align}
\Psi_k^\ast(X,h)
=&\ 
\sum_{XL^{-4kA}<n^k\le X}\Lambda(n)\Lambda(n^k+h)
+O(X^{1/k}L^{-3A})\notag\\
\label{Fourier}
=&\ 
\int_0^1S_1(\alpha)\overline{S_k(\alpha)}e(-h\alpha)d\alpha+O(X^{1/k}L^{-3A})
\end{align}
for any $h\le H$.
We use the Farey dissection given by
\begin{gather*}
P=L^B,\quad
Q=H^{1/2},\quad
R=XP^{-4},\quad
I=\left[\frac{1}{Q},1+\frac{1}{Q}\right],\\
\Maq=\left[\frac{a}{q}-\frac{1}{qQ},\frac{a}{q}+\frac{1}{qQ}\right],\quad
\Naq=\left[\frac{a}{q}-\frac{1}{qR},\frac{a}{q}+\frac{1}{qR}\right],\\
\M=\coprod_{q\le P}\puni_{a\ppmod{q}}\Naq,\quad
\m=I\setminus\M.
\end{gather*}
Then by the integral expression (\ref{Fourier}),
we have
\begin{gather*}
\sum_{U<h\le U+H}
\left|\Psi_k^\ast(X,h)-\mathfrak{S}_k^\ast(h,P)X^{1/k}\right|^2\\
\ll\sum_{U<h\le U+H}
\left|\int_\M S_1(\alpha)\overline{S_k(\alpha)}e(-h\alpha)d\alpha
-\mathfrak{S}_k^\ast(h,P)X^{1/k}\right|^2\\
+\sum_{U<h\le U+H}
\left|\int_\m S_1(\alpha)\overline{S_k(\alpha)}e(-h\alpha)d\alpha\right|^2
+HX^{2/k}L^{-6A},
\end{gather*}
which is
\[=\sum_\M+\sum_\m+HX^{2/k}L^{-6A},\ \text{say}.\]

\begin{rem}
At first sight, the Farey arcs $\Maq$ are not used in the course of the proof.
However, we discuss the arcs $\Maq$ for the proof of Theorem \ref{90_moment}.
See \cite[Section 5]{PP_Goldbach}.
\end{rem}

\section{Preliminary lemmas}
We first approximate the trigonometric polynomial $S_k(\alpha)$
in a standard way.

\begin{lem}
\label{S_approx}
We have
\[S_k\left(\frac{a}{q}+\eta\right)
=\frac{C_k(a,q)}{\phi(q)}V_k(\eta)
+O\left(q(1+|\eta|X)X^{1/k}P^{-16}\right)\]
for any $k\ge1$.
\end{lem}

\begin{proof}
If $q>P^{16}$, then this lemma is reduced to the trivial estimate since
\[
q(1+|\eta|X)X^{1/k}P^{-16}\gg X^{1/k}.
\]
Hence we assume $q\le P^{16}$ without loss of generality.  
We have
\begin{equation}
\label{pre_Siegel_Walfisz}
S_k\left(\frac{a}{q}+\eta\right)
=
\psum_{m\ppmod{q}}
e\left(\frac{am^k}{q}\right)
\sum_{n\equiv m\ppmod{q}}
\Lambda(n)e(n^k\eta)
+O(L^2).
\end{equation}
By the Siegel-Walfisz theorem \cite[Corollary 5.29]{Iwaniec_Kowalski},
we have
\[
\sum_{n\equiv m\ppmod{q}}
\Lambda(n)e(n^k\eta)
=
\frac{1}{\phi(q)}V_k(\eta)+O\left((1+|\eta|X)X^{1/k}P^{-16}\right).
\]
Substituting this into (\ref{pre_Siegel_Walfisz}),
we obtain the lemma.
\end{proof}

We next recall some basic facts on the complete exponential sums.
For the detailed proofs and discussions,
see Section 4 and 5 of \cite{Bauer_thesis}.

\begin{lem}[{\cite[Lemma 4.3 (b)]{Bauer_thesis}}]
\label{multiplicative_A}
Suppose that $(q_1,q_2)=1$.
Then
\[A_k(h,q_1q_2)=A_k(h,q_1)A_k(h,q_2).\]
\end{lem}
\begin{proof}
Immediately follows from the Chinese remainder theorem.
\end{proof}

\begin{lem}[{\cite[Lemma 4.4 (a)]{Bauer_thesis}}]
\label{explicit_A}
For any prime $p$,
we have
\[A_k(h,p)=p\cdot w_k(h,p)-\phi(p),\]
where $w_k(h,p)$ is given by $(\ref{w_def})$.
\end{lem}
\begin{proof}
Immediately follows from the orthogonality.
\end{proof}

\section{The major arcs}
In this section,
we shall evaluate the integral over the major arcs.
We have
\begin{gather*}
\int_\M S_1(\alpha)\overline{S_k(\alpha)}e(-h\alpha)d\alpha\\
=\sum_{q\le P}\psum_{a\ppmod{q}}e\left(-\frac{ah}{q}\right)
\int_{|\eta|\le1/qR}
S_1\left(\frac{a}{q}+\eta\right)
\overline{S_k\left(\frac{a}{q}+\eta\right)}
e(-h\eta)d\eta,
\end{gather*}
which we denote by
\[=\sum_{q\le P}\psum_{a\ppmod{q}}e\left(-\frac{ah}{q}\right)J_{a,q}(h).\]
We approximate each integral $J_{a,q}(h)$
by decomposing into the following parts:
\[J_{a,q}(h)=A_{a,q}(h)+B_{a,q}(h)+C_{a,q}(h)+I_{a,q}(h),\]
where
\begin{gather*}
A_{a,q}(h)=
\int_{|\eta|\le1/qR}
S_1\left(\frac{a}{q}+\eta\right)\overline{R_k(\eta,a,q)}e(-h\eta)d\eta,\\
B_{a,q}(h)=
\frac{\overline{C_k(a,q)}}{\phi(q)}
\int_{|\eta|\le1/qR}R_1(\eta,a,q)\overline{V_k(\eta)}e(-h\eta)d\eta,\\
C_{a,q}(h)=-\frac{\mu(q)\overline{C_k(a,q)}}{\phi(q)^2}
\int_{1/qR<|\eta|\le1/2}V_1(\eta)\overline{V_k(\eta)}e(-h\eta)d\eta,\\
I_{a,q}(h)=\frac{\mu(q)\overline{C_k(a,q)}}{\phi(q)^2}
\int_{|\eta|\le1/2}V_1(\eta)\overline{V_k(\eta)}e(-h\eta)d\eta.
\end{gather*}
We shall prove the estimates
\begin{equation}
\label{ABC_estimate}
A_{a,q}(h), B_{a,q}(h), C_{a,q}(h)\ll X^{1/k}P^{-2}L^{-2A},
\end{equation}
and the asymptotic formula
\begin{equation}
\label{asymp_I}
\sum_{q\le P}\psum_{a\ppmod{q}}e\left(-\frac{ah}{q}\right)I_{a,q}(h)=
\mathfrak{S}_k^\ast(h,P)X^{1/k}+O(X^{1/k}L^{-2A}).
\end{equation}

We start with $A_{a,q}(h)$.
Since $S_1(\alpha)\ll X$,
we have
\[
A_{a,q}(h)\ll
X\int_{|\eta|\le1/qR}\left|R_k(\eta,a,q)\right|d\eta.
\]
Then Lemma \ref{S_approx} gives
\[
A_{a,q}(h)
\ll
X^{2+1/k}R^{-2}P^{-16}
\ll
X^{1/k}P^{-2}L^{-2A}.
\]
This proves (\ref{ABC_estimate}) for $A_{a,q}(h)$.
The integral $B_{a,q}(h)$ can be estimated similarly.

We next estimate the integral $C_{a,q}(h)$.
Note that for $|\eta|\le 1/2$
\[
V_1(\eta)\ll|\eta|^{-1},\quad
V_k(\eta)\ll\frac{L^{4kA}}{X^{1-1/k}|\eta|}.
\]
For the proof of these estimates, see \cite[Corollary 8.11]{Iwaniec_Kowalski}.
Thus we have
\[
C_{a,q}(h)
\ll
\frac{X^{1/k-1}L^{4kA}}{\phi(q)}
\int_{1/qR<|\eta|\le1/2}\frac{d\eta}{|\eta|^2}
\ll
RX^{1/k-1}L^{5kA}
\ll
X^{1/k}P^{-2}L^{-2A}.
\]
This proves (\ref{ABC_estimate}) for $C_{a,q}(h)$.

Finally we prove the asymptotic formula (\ref{asymp_I}).
Clearly
\[
\sum_{q\le P}\psum_{a\ppmod{q}}e\left(-\frac{ah}{q}\right)I_{a,q}(h)
=
\mathfrak{S}_k^\ast(h,P)
\int_{|\eta|\le1/2}V_1(\eta)\overline{V_k(\eta)}e(-h\eta)d\eta.
\]
By the orthogonality of additive characters,
we have
\[
\int_{|\eta|\le1/2}V_1(\eta)\overline{V_k(\eta)}e(-h\eta)d\eta
=
X^{1/k}+O(X^{1/k}L^{-4A}).
\]
Since Lemma \ref{multiplicative_A} and \ref{explicit_A} implies
\begin{equation}
\label{truncated_estimate}
\mathfrak{S}_k^\ast(h,P)
\ll
\sum_{q\le P}\frac{\mu^2(q)k^{\nu(q)}q}{\phi(q)^2}
\ll\prod_{p\le P}\left(1+\frac{kp}{(p-1)^2}\right)\ll L^k,
\end{equation}
we obtain (\ref{asymp_I}).

By
(\ref{ABC_estimate})
and
(\ref{asymp_I}),
we arrive at
\begin{equation}
\label{major_arc_estimate}
\sum_\M\ll HX^{2/k}L^{-4A}.
\end{equation}
This completes the evaluation of the major arcs.

\section{Lemmas for the minor arcs}
\label{lemma_minor}
The remaining task is to estimate the integral over the minor arcs.
In this section,
we prepare some lemmas for the minor arc estimate.

As we mentioned before,
we shall reduce our minor arc estimate
to the corresponding estimate for the twin prime problem.
This minor arc estimate was obtained by Mikawa \cite{Mikawa_twin}
or by Perelli and Pintz \cite{PP_Goldbach}.
Their result can be stated as:
\begin{thm}
\label{90_moment}
Let $0\le U\le X$, $H\le V\ll X$
and assume the above setting.
Then
\[\sum_{U<h\le U+V}
\left|\int_\m\left|S_1(\alpha)\right|^2e(h\alpha)d\alpha\right|^2
\ll
VX^2L^{-32kKA}\]
for sufficiently large $B\ge B_0(k,A)$.
\end{thm}
\noindent%
Since our Farey dissection is given in the same manner
as Perelli and Pintz \cite{PP_Goldbach} used,
it is more direct to apply the proof of Perelli and Pintz \cite{PP_Goldbach}.
Note that the admissible range of $H$ obtained
in \cite{Mikawa_twin, PP_Goldbach} is $X^{1/3+\epsilon}\le H\ll X$,
which is much stronger than we need here.

As for the reduction of our minor arc estimate to Theorem \ref{90_moment},
we use the idea of Mikawa and Peneva \cite{Mikawa_Peneva}.
In order to carry out their technique with general exponent $k$,
we need some lemmas
which correspond to Lemma 3 in \cite{Mikawa_Peneva}.

We use the Ces\`aro weight
\[w(h)=\max\left(1-\frac{|h|}{2H},0\right),\]
which appears as the coefficient of the Fej\'er kernel%
\footnote{
Recall that we assume $H$ is a positive integer.
}
\[F(\alpha)=\sum_{|h|\le 2H}w(h)e(h\alpha).\]
Recall that the Fej\'er kernel is non-negative
since
\begin{align*}
F(\alpha)=\frac{1}{2H}\left(\frac{\sin2\pi H\alpha}{\sin\pi\alpha}\right)^2.
\end{align*}
For any real numbers $M$ and $M'$ satisfying
\[1\le M<M'\le 2M,\quad M^{1-1/k}\le H,\]
we let
\[\Phi(\alpha):=
\dsum_{M<m_1^k,m_2^k\le M'}
w(m_1^k-m_2^k)e((m_1^k-m_2^k)\alpha).\]
Our first two lemmas are on some basic properties of this kernel $\Phi(\alpha)$.
\begin{lem}
\label{Phi_positive}
For any real number $\alpha$, we have $\Phi(\alpha)\ge0$.
\end{lem}

\begin{proof}
We have
\begin{align*}
\Phi(\alpha)
=&
\dsum_{M<m_1^k,m_2^k\le M'}
\left(\int_{-1/2}^{1/2}F(\eta)e((m_1^k-m_2^k)\eta)d\eta\right)
e((m_1^k-m_2^k)\alpha)\\
=&\int_{-1/2}^{1/2}
\left|\sum_{M<m^k\le M'}e(m^k(\alpha+\eta))\right|^2
F(\eta)d\eta\ge0.
\end{align*}
This gives the lemma.
\end{proof}

\begin{lem}
\label{Phi_at_0_estimate}
Suppose that $M^{1-1/k}\le H$.
Then $\Phi(0)\ll HM^{2/k-1}$.
\end{lem}

\begin{proof}
We have
\[
\Phi(0)
=
\dsum_{M<m_1^k,m_2^k\le M'}w(m_1^k-m_2^k)
\ll
\dsum_{\substack{
M<m_1^k,m_2^k\le M'\\
m_1\ge m_2
}}w(m_1^k-m_2^k).
\]
Here we introduce two new variables
\[d=m_1-m_2,\quad m=m_2.\]
By the definition of $w(h)$,
we find that
\begin{equation}
\label{support_w}
|h|>2H\ 
\Longrightarrow\ 
w(h)=0.
\end{equation}
Hence our new variables $d$ and $m$ are restricted by
\[0\le m_1^k-m_2^k=(m+d)^k-m^k=d(km^{k-1}+\cdots+d^{k-1})\le2H.\]
In particular, we can restrict the variable $d$ by
\[0\le d\le HM^{1/k-1}.\]
Therefore%
\footnote{
Notice that $HM^{1/k-1}\gg1$ by the assumption $M^{1-1/k}\le H$.
}
\[
\Phi(0)
=
\dsum_{M<m_1^k,m_2^k\le M'}w(m_1^k-m_2^k)
\ll
\sum_{0\le d\le HM^{1/k-1}}\sum_{m\ll M^{1/k}}1
\ll HM^{2/k-1}.
\]
This completes the proof.
\end{proof}

Next we want
to reduce the degree of the polynomial in the definition of $\Phi(\alpha)$
by using the Weyl differencing.
We start with recalling the Weyl differencing in the form we use.
We introduce some notation following Bauer \cite[Section 3]{Bauer}.
We use the forward difference operator $\Delta(\ast;u_1,\dots,u_k)$
on the polynomial ring $\mathbb{R}[X]$
which is defined inductively by
\begin{gather*}
\Delta(f(X);u)=f(X+u)-f(X),\\
\Delta(f(X);u_1,\dots,u_k,u_{k+1})=
\Delta(\Delta(f(X);u_1,\dots,u_k);u_{k+1})
\end{gather*}
for integers $u,u_1,\dots,u_k$ and a polynomial $f(X)\in\mathbb{R}[X]$.
We also use the operator $\nabla(\ast;u_1,\dots,u_k)$
on the ring of real-valued arithmetic functions which is defined inductively by
\[\nabla(g(n);u)=g(n+u)\cdot g(n),\]
\[\nabla(g(n);u_1,\dots,u_k,u_{k+1})=
\nabla(\nabla(g(n);u_1,\dots,u_k);u_{k+1})\]
for integers $u,u_1,\dots,u_k$
and a real-valued arithmetic function $g(n)$ defined on $\mathbb{Z}$.
Then the Weyl differencing is the following.
\begin{lem}[Weyl differencing]
Let $k\ge2$ be an integer, $X\ge1$, $K=2^{k-1}$,
and $I$ be an interval of length $\le X$.
Then we have
\begin{gather*}
\left|
\sum_{m\in I}
g(m)e(f(m))
\right|^K\\
\ll
X^{K-k}\msum_{|u_1|,\dots,|u_{k-1}|\le X}
\sum_{m:(\ast)}
\nabla(g(m);u_1,\dots,u_{k-1})
e\left(\Delta(f(m);u_1,\dots,u_{k-1})\right)
\end{gather*}
where the condition $(\ast)$ on the summation variable $m$ is given by
\begin{gather*}
(\ast):\ \ 
\forall\,\mathcal{U}\subset\{u_1,\dots,u_{k-1}\},\ \ 
m+\sum_{u\in\mathcal{U}}u\in I.
\end{gather*}
\end{lem}

By using the Weyl differencing, we have
\begin{lem}
\label{Phi_Weyl}
Let $K=2^{k-1}$ and suppose $M^{1/k-1}\le H$.
Then we have
\begin{align*}
\Phi(\alpha)^{K/2}\ll
H^{K/2-1}M^{K/k-K/2}
\Theta(\alpha)
+H^{K/2}M^{K/k-K/2-1/k}+M^{K/2k},
\end{align*}
where the trigonometric polynomial $\Theta(\alpha)$ is given by
\[\Theta(\alpha)=\sum_{1\le|h|\ll H}c(h)e(h\alpha),\]
and its coefficients satisfy $c(h)\ll\tau_k(|h|)$ for all $h\neq0$.
\end{lem}

\begin{proof}
We consider two cases $k=2$ and $k\ge3$ separately.
For the case $k=2$,
\begin{align*}
\Phi(\alpha)
=&\ \dsum_{\substack{M<m_1^2,m_2^2\le M'\\m_1\neq m_2}}
w(m_1^2-m_2^2)e((m_1^2-m_2^2)\alpha)+O(M^{1/k})\\
=&\ 
\sum_{1\le|h|\ll H}c(h)e(h\alpha)+O(M^{1/k}),
\end{align*}
where
\[
c(h)
=
\dsum_{\substack{M<m_1^2,m_2^2\le M'\\m_1^2-m_2^2=h}}w(m_1^2-m_2^2)
\ll
\dsum_{(m_1-m_2)(m_2+m_2)=h}1
\ll
\tau_2(|h|)
\]
for all $h\neq0$.
This completes the proof for the case $k=2$.

For the case $k\ge3$,
we have
\begin{align*}
\Phi(\alpha)
=&\dsum_{\substack{M<m_1^k,m_2^k\le M'\\m_1\neq m_2}}
w(m_1^k-m_2^k)e((m_1^k-m_2^k)\alpha)+O(M^{1/k})\\
=&\ 
2\Re\dsum_{\substack{M<m_1^k,m_2^k\le M'\\m_1>m_2}}
w(m_1^k-m_2^k)e((m_1^k-m_2^k)\alpha)+O(M^{1/k}).
\end{align*}
By (\ref{support_w}),
we can rewrite the last expression as
\[
\ll
\left|
\sum_{d\le HM^{1/k-1}}
\sum_{\substack{M<m^k,(m+d)^k\le M'\\\Delta(m^k;d)\le2H}}
w\left(\Delta(m^k;d)\right)
e\left(\Delta(m^k;d)\alpha\right)
\right|+M^{1/k}.
\]
Applying H\"older's inequality,
we have
\begin{equation}
\label{Phi0_intro}
\Phi(\alpha)^{K/2}
\ll(HM^{1/k-1})^{K/2-1}\Phi_0(\alpha)+M^{K/2k},
\end{equation}
where
\[
\Phi_0(\alpha)=
\sum_{d\le HM^{1/k-1}}
\left|
\sum_{\substack{M<m^k,(m+d)^k\le M'\\\Delta(m^k;d)\le2H}}
w\left(\Delta(m^k;d)\right)
e\left(\Delta(m^k;d)\alpha\right)
\right|^{K/2}.
\]
We carry out the Weyl differencing here
and obtain
\begin{align*}
\Phi_0(\alpha)
\ll&\ 
M^{K/2k-(k-1)/k}
\sum_{d\le HM^{1/k-1}}
\msum_{|u_1|,\dots,|u_{k-2}|\le M^{1/k}}
\Phi_1(d,u_1,\dots,u_{k-2})\\
\ll&\ 
M^{K/2k+(1/k-1)}
\sum_{d\le HM^{1/k-1}}
\msum_{1\le |u_1|,\dots,|u_{k-2}|\le M^{1/k}}
\Phi_1(d,u_1,\dots,u_{k-2})\\
&\hspace{6cm}
+M^{K/2k+(1/k-1)}\cdot HM^{-1/k},
\end{align*}
where
\begin{gather*}
\Phi_1(d,u_1,\dots,u_{k-2})
=
\sum_{m:(\ast1)}
g(m;d,u_1,\dots,u_{k-2})
e\left(\Delta(m^k;d,u_1,\dots,u_{k-2})\alpha\right),\\
g(m;d,u_1,\dots,u_{k-2})
=
\nabla\left(w(\Delta(m^k;d));u_1,\dots,u_{k-2}\right),
\end{gather*}
and the summation condition $(\ast1)$ is given by
\begin{gather*}
(\ast1):\ \ 
\forall\,\mathcal{U}\subset\{u_1,\dots,u_{k-2}\},\ \ 
\left\{
\substack{
\displaystyle
M<
\left(m+\sum_{u\in\mathcal{U}}u\right)^k
\le M'\\
\displaystyle
M<
\left(m+d+\sum_{u\in\mathcal{U}}u\right)^k
\le M'\\[2mm]
\displaystyle
\Delta((m+\sum_{u\in\mathcal{U}}u)^k;d)\le2H}
\right\}.
\end{gather*}
We group the terms according to the values
\[\Delta(m^k;d,u_1,\dots,u_{k-2}).\]
In order to do this,
we observe
\begin{gather*}
\Delta(m^k;d,u_1,\dots,u_{k-2})
=\frac{k!}{2}du_1\cdots u_{k-2}
\left(2m+d+u_1+\cdots+u_{k-2}\right),\\
\left|\Delta(m^k;d,u_1,\dots,u_{k-2})\right|
\ll
H,\quad
g(m;d,u_1,\dots,u_{k-2})\ll 1.
\end{gather*}
Hence for any nonnegative integer $h$
the equation
\[\Delta(m^k;d,u_1,\dots,u_{k-2})=h\]
has at most $\ll\tau_k(|h|)$ solutions $(m,d,u_1,\dots,u_{k-2})$.
Therefore we obtain
\[\Phi_0(\alpha)\ll 
M^{K/2k+(1/k-1)}
\sum_{1\le|h|\ll H}c(h)e(h\alpha)+HM^{K/2k-1},\]
where
\[
c(h):=
\msum_{\substack{
d\le HM^{1/k-1}\\
1\le |u_1|,\dots,|u_{k-2}|\le M^{1/k}\\[1mm]
m:(\ast1)\\
\Delta(m^k;d,u_1,\dots,u_{k-2})=h
}}
g(m;d,u_1,\dots,u_{k-2})
\ll
\tau_k(|h|).
\]
Substituting this expression into (\ref{Phi0_intro}),
we arrive at
\begin{equation*}
\begin{split}
\Phi(\alpha)^{K/2}\ll&\ 
H^{K/2-1}M^{K/k-K/2}
\sum_{1\le|h|\ll H}c(h)e(h\alpha)\\
&\hspace{3cm}
+H^{K/2}M^{K/k-K/2-1/k}+M^{K/2k}.
\end{split}
\end{equation*}
This completes the proof of the lemma.
\end{proof}

\section{The minor arcs}
Now we proceed to the estimate for the minor arcs.
We first subdivide the sum over prime powers dyadically:
\begin{equation}
\label{dyadic_minor}
\sum_\m=
\sum_{U<h\le U+H}\left|\int_\m
S_1(\alpha)\overline{S_k(\alpha)}e(-h\alpha)d\alpha\right|^2
\ll
L^2\sup_{\substack{XL^{-4kA}<M\le X\\M<M'\le2M}}\sum_{\m,M},
\end{equation}
where
\[\sum_{\m,M}=
\sum_{U<h\le U+H}\left|\int_\m
S_1(\alpha)\overline{S_k(\alpha,M)}e(-h\alpha)d\alpha\right|^2,\]
\[S_k(\alpha,M)=\sum_{M<m^k\le M'}\Lambda(m)e(m^k\alpha).\]
Next we introduce the weights $w(h)$ into the sum $\sum_{\m,M}$.
Then%
\footnote{
Recall that we assume that $U$ is a positive integer.
}
\begin{align*}
\sum_{\m,M}\ll&
\sum_{U<h\le U+H}\left|\int_\m
S_1(\alpha)\overline{S_k(\alpha,M)}e(-h\alpha)d\alpha\right|^2\\
\ll&
\sum_{|h|\le 2H}w(h)\left|\int_\m
S_1(\alpha)\overline{S_k(\alpha,M)}e(-(U+h)\alpha)d\alpha\right|^2.
\end{align*}
We expand the square and take summation over $h$.
Then we have
\[\sum_{\m,M}\ll\int_\m\int_\m
\left|S_1(\alpha)S_k(\alpha, M)S_1(\beta)S_k(\beta, M)\right|
F(\alpha-\beta)d\alpha d\beta.\]
By the inequality of the arithmetic and geometric means,
we have
\[
\left|S_1(\alpha)S_k(\alpha, M)S_1(\beta)S_k(\beta, M)\right|
\ll
\left|S_1(\alpha)S_k(\beta, M)\right|^2
+
\left|S_1(\beta)S_k(\alpha, M)\right|^2.\]
Therefore we have
\begin{align*}
\sum_{\m,M}
\ll&
\int_\m\int_\m
\left|S_1(\alpha)\right|^2\left|S_k(\beta, M)\right|^2
F(\alpha-\beta)d\alpha d\beta\\
\ll&
\int_{-1/2}^{1/2}\int_\m
\left|S_1(\alpha)\right|^2\left|S_k(\alpha+\beta, M)\right|^2
F(\beta)d\alpha d\beta.
\end{align*}
Now we expand the square
\[\left|S_k(\alpha+\beta,M)\right|^2,\]
and interchange the order of integration and summation.
Then we have
\begin{align*}
\sum_{\m,M}\ll&
\dsum_{M<m_1^k,m_2^k\le M'}\Lambda(m_1)\Lambda(m_2)\\
&\times
\int_{-1/2}^{1/2}\int_\m
\left|S_1(\alpha)\right|^2
F(\beta)
e((m_1^k-m_2^k)(\alpha+\beta))d\alpha d\beta\\
\ll&\ 
L^2
\dsum_{M<m_1^k,m_2^k\le M'}
\left|\int_{-1/2}^{1/2}
F(\beta)e((m_1^k-m_2^k)\beta)d\beta\right|\\
&\hspace{4.2cm}
\times\left|\int_\m
\left|S_1(\alpha)\right|^2
e((m_1^k-m_2^k)\alpha)d\alpha\right|\\
=&\ 
L^2
\dsum_{M<m_1^k,m_2^k\le M'}
w(m_1^k-m_2^k)
\left|\int_\m
\left|S_1(\alpha)\right|^2
e((m_1^k-m_2^k)\alpha)d\alpha\right|.
\end{align*}
By the Cauchy-Schwarz inequality, we have
\begin{equation}
\label{Phi_J}
\sum_{\m,M}
\ll
\Phi(0)^{1/2}J^{1/2}L^2,
\end{equation}
where
\[J=
\dsum_{M<m_1^k,m_2^k\le M'}
w(m_1^k-m_2^k)
\left|\int_\m
\left|S_1(\alpha)\right|^2
e((m_1^k-m_2^k)\alpha)d\alpha\right|^2.\]
We estimate this sum $J$.
Expanding the square and interchanging the order of summation and integration,
we have
\[J\ll
\int_\m\int_\m
\left|S_1(\alpha)\right|^2\left|S_1(\beta)\right|^2
\Phi(\alpha-\beta)d\alpha d\beta.\]
We apply H\"older's inequality and obtain
\begin{equation}
\label{J_intro}
J\ll
\left(
\int_\m\left|S_1(\alpha)\right|^2d\alpha
\right)^{2(K-2)/K}
J_0^{2/K}
\ll (XL)^{2(K-2)/K}J_0^{2/K},
\end{equation}
where
\begin{equation}
\label{J0_definition}
J_0=
\int_\m\int_\m
\left|S_1(\alpha)\right|^2\left|S_1(\beta)\right|^2
\Phi(\alpha-\beta)^{K/2}d\alpha d\beta.
\end{equation}
Now we substitute Lemma \ref{Phi_Weyl} into (\ref{J0_definition}).
Then we find that
\begin{align*}
J_0\ll&\ 
H^{K/2-1}M^{K/k-K/2}
\int_\m\int_\m
\left|S_1(\alpha)\right|^2\left|S_1(\beta)\right|^2
\sum_{1\le|h|\ll H}c(h)e(h(\alpha-\beta))d\alpha d\beta\\
&\hspace{0.5cm}+\left(H^{K/2}M^{K/k-K/2-1/k}+M^{K/2k}\right)
\left(\int_0^1\left|S_1(\alpha)\right|^2d\alpha\right)^2\\
\ll&\ 
H^{K/2-1}M^{K/k-K/2}
\sum_{1\le |h|\ll H}c(h)
\left|\int_\m\left|S_1(\alpha)\right|^2e(h\alpha)d\alpha\right|^2\\
&\hspace{0.5cm}+\left(H^{K/2}M^{K/k-K/2-1/k}+M^{K/2k}\right)(XL)^2.
\end{align*}
By Theorem \ref{90_moment},
we have
\begin{gather*}
\sum_{1\le |h|\ll H}c(h)
\left|\int_\m\left|S_1(\alpha)\right|^2e(h\alpha)d\alpha\right|^2\\
\ll
\left(\sum_{1\le h\ll H}\tau_k(h)^2
\left(\int_\m\left|S_1(\alpha)\right|^2d\alpha\right)^2\right)^{1/2}
\left(\sum_{1\le h\ll H}
\left|\int_\m\left|S_1(\alpha)\right|^2e(h\alpha)d\alpha\right|^2\right)^{1/2}\\
\ll\left(HX^2L^{2k^2}\right)^{1/2}\Big(HX^2L^{-32kKA}\Big)^{1/2}
\ll HX^2L^{-15kKA}.
\end{gather*}
Therefore
we obtain
\[
J_0
\ll
H^{K/2}X^2M^{K/k-K/2}L^{-15kKA}+M^{K/2k}(XL)^2.
\]
We substitute this estimate into (\ref{J_intro}).
Then
\begin{equation}
\label{J_estimate}
J\ll
HX^2M^{-1+2/k}L^{-28kA}
+M^{1/k}(XL)^2.
\end{equation}
We combine (\ref{Phi_J}), (\ref{J_estimate}), and Lemma \ref{Phi_at_0_estimate}.
Then we arrive at
\[
\sum_{\m,M}
\ll
\left(HM^{2/k-1}J\right)^{1/2}L^2
\ll
HXM^{-1+2/k}L^{-12kA}
+\,H^{1/2}XM^{3/2k-1/2}L^3.
\]
Since the assumptions
\[X^{1-1/k+\epsilon}\le H\le X,\quad XL^{-4kA}<M\le X\]
imply
\[H^{1/2}XM^{3/2k-1/2}L^3
\ll HX^{2/k}L^{-12kA},\]
we have
\[\sum_{\m,M}\ll HX^{2/k}L^{-12kA}.\]
Substituting this estimate into (\ref{dyadic_minor}),
we arrive at
\[\sum_\m\ll HX^{2/k}L^{-4A}\]
as desired.
This completes the proof of Theorem \ref{main_mean}.

\section{Completion of the proof}
We need to approximate the truncated series $\mathfrak{S}_k^\ast(h,P)$
by the full series $\mathfrak{S}_k^\ast(h)$.
This task turns out to be difficult.
Fortunately, Kawada \cite{Kawada} had already
developed the techniques on the completion of the singular series.
We just refer a variant of Kawada's result.
\begin{lem}
\label{completion_singular}
Assume $X^{\epsilon}\le H\le X$.
Then we have
\[
\mathfrak{S}_k^\ast(h,P)
=
\mathfrak{S}_k^\ast(h)+O\left(L^{-A}\right)
\]
for all but $\ll HL^{-A}$ integers $h\in[1,H]\cap\mathbb{H}_k$.
\end{lem}
\begin{proof}
This can be proven by Kawada's method \cite[Corollary 1]{Kawada}.
\end{proof}

\begin{rem}
Since Lemma \ref{completion_singular} is not the short interval version,
i.e. our range of $h$ is not $[X,X+H]$ but $[1,H]$,
there is no need to assume $X^{1/2+\epsilon}\le H$.
Our assumption $X^\epsilon\le H\le X$ just assures $\log X\asymp \log H$.
\end{rem}

We can now prove Theorem \ref{main_thm}.
By Theorem \ref{main_mean} with $U=0$,
we have
\begin{gather*}
\#\Set{h\in[1,H]\cap\mathbb{H}_k|
\left|\Psi_k^\ast(X,h)-\mathfrak{S}_k^\ast(h,P)X^{1/k}\right|>X^{1/k}L^{-A}}\\
\ll \frac{HX^{2/k}L^{-4A}}{X^{2/k}L^{-2A}}\ll HL^{-A}.
\end{gather*}
Therefore we have
\[
\Psi_k^\ast(X,h)
=
\mathfrak{S}_k^\ast(h,P)X^{1/k}+O\left(X^{1/k}L^{-A}\right)
\]
for all but $\ll HL^{-A}$ integers $h\in[1,H]\cap\mathbb{H}_k$.
Now Lemma \ref{completion_singular} implies that
\[
\Psi_k^\ast(X,h)
=
\mathfrak{S}_k^\ast(h)X^{1/k}+O\left(X^{1/k}L^{-A}\right)
\]
with $\ll HL^{-A}$ additional exceptions.
This completes the proof of Theorem \ref{main_thm}.

\section{Some remarks}
We give two remarks in order to compare the method of Bauer \cite{Bauer} and ours.

\begin{rem}
\label{disadvantage}
We first recall that for the conjugate equation (\ref{conjugate_Hua}),
we can use the restricted counting function
\[
\tilde{R}_k(N)=
\sum_{\substack{m+n^k=N\\X-Y<m\le X\\Y/2^k<n^k\le Y/2^k+Y}}
\Lambda(m)\Lambda(n)
\]
instead of the full counting function
\[
R_k(N)=
\sum_{m+n^k=N}\Lambda(m)\Lambda(n),
\]
where $Y$ is some parameter smaller than $X$.
By using the prime number theorem in short intervals,
we can obtain some result for $\tilde{R}_k(N)$ even better than for $R_k(N)$
if $Y$ is substantially smaller than $X$.
In the paper \cite{Bauer},
Bauer stated his result with $\tilde{R}_k(N)$ and he obtained the admissible range
\[
Y^{1-\frac{1}{2k}+\epsilon}\le H\le Y,\quad
X^{\frac{7}{12}+\epsilon}\le Y\le X.
\]
Unfortunately,
it seems impossible to combine our method with this restriction trick.
Thus our Theorem \ref{main_thm} is an improvement
only for the full counting function $R_k(N)$.
This is a disadvantage of our method comparing with the method of Bauer.
On the other hand, note that the information of $R_k(N)$
cannot be restored from that of $\tilde{R}_k(N)$
if $Y$ is of the size $o(X)$.
\end{rem}

\begin{rem}
\label{restriction_fails}
As for the equation (\ref{Hua_eq}) with $h$ in some neighborhood of $h=2$,
the restriction trick in Remark \ref{disadvantage} does not work well.
Since if we introduce the restriction $n^k\le Y$
to the sum $\Psi_k^\ast(X,h)$, then the resulting sum is
\[
\sum_{\substack{n^k\le X\\n^k\le Y}}\Lambda(n)\Lambda(n^k+h)
=
\sum_{n^k\le Y}\Lambda(n)\Lambda(n^k+h)
=
\Psi_k^\ast(Y,h)
\]
so that the restriction trick just replace the main variable $X$ by $Y$,
which violates the desired situation ``the larger $X$ with the smaller $h$''.
However, as Perelli and Pintz \cite[Theorem 3]{PP_Goldbach} mentioned,
we can obtain the result for the counting function $\Psi_k^\ast(Y,h)$
with $h\in[X,X+H]$.
This result is rather motivated by the problem asking the expression
\[
N=p'-p^k
\]
of a given integer $N$, which has slightly different interest from our problem asking
the distribution of prime vs. prime power pairs.
\end{rem}

\begin{center}
\textbf{Acknowledgements.}
\end{center}

The author would like to thank
Prof. Kohji Matsumoto,
Prof. Hiroshi Mikawa,
Prof. Koichi Kawada,
and
Prof. Alberto Perelli
for their invaluable comments and suggestions.
This work was supported by Grant-in-Aid for JSPS Research Fellow
(Grant Number: JP16J00906).

\vspace{3mm}

\begin{flushleft}
{\footnotesize
{\sc
Graduate School of Mathematics, Nagoya University,\\
Chikusa-ku, Nagoya 464-8602, Japan.
}

{\it E-mail address}: {\tt m14021y@math.nagoya-u.ac.jp}
}
\end{flushleft}

\end{document}